\documentclass[a4paper, 12pt]{amsart}
\usepackage[english]{babel}
\usepackage{amssymb,amstext}
\usepackage{amsthm}
\usepackage[latin1]{inputenc} %text mit umlauten
\usepackage[T1]{fontenc}      %ec-fonts
\usepackage{graphicx, xcolor}
\usepackage{textcomp}    
\date{\today}

\usepackage[top=1.5in, bottom=1.3in, left=1.1in, right=1.1in]{geometry}

\newtheorem{theorem}{Theorem}
\newtheorem{corollary}[theorem]{Corollary}

\newtheorem{lemma}[theorem]{Lemma}
\newtheorem{proposition}[theorem]{Proposition}
\newtheorem{definition}[theorem]{Definition}
\newtheorem{assumption}[theorem]{Assumption}

\theoremstyle{definition}
\newtheorem{remark}[theorem]{Remark}

\newcommand{\xad}{x_\alpha^\delta}

\newcommand{\xdag}{x^\dagger}
\newcommand{\xdags}{\underline x^\dagger}

\newcommand{\yd}{y^\delta}
\newcommand{\range}{\mathcal{R}}

\newcommand{\supp}{\mathop{\mathrm{supp}}}
\newcommand{\sgn}{\mathop{\mathrm{sgn}}}

\newcommand{\R}{\mathbb{R}}
\newcommand{\N}{\mathbb{N}}
\newcommand{\Z}{\mathbb{Z}}

\newcommand{\1}{\ell^1(\N)}
\newcommand{\2}{\ell^2(\N)}
\newcommand{\3}{\ell^\infty(\N)}

\newcommand{\M}{\mathcal{M}}
\newcommand{\Mqw}{{\ensuremath{\mathcal{M}_{q,w}}}}
\newcommand{\Reg}{{\ensuremath{\mathcal{R}}}}
\newcommand{\Rqw}{{\ensuremath{\mathcal{R}_{q,w}}}}
\newcommand{\Rzw}{{\ensuremath{\mathcal{R}_{0}}}}
\newcommand{\Row}{{\ensuremath{\mathcal{R}_{1}}}}
\newcommand{\Rtw}{{\ensuremath{\mathcal{R}_{2}}}}
\newcommand{\RT}{{R}}

\newcommand\domain[1]{{{\mathcal{D}}(#1)}}

\title[On basis smoothness in sparsity regularization]{\bf On the interplay of basis smoothness and specific range conditions occurring in sparsity regularization}

\author{Stephan W. Anzengruber$^*$}
\address{Department of Mathematics, Chemnitz University of Technology,
  09107 Chemnitz,  Germany}
\email{\tt stephan.anzengruber@mathematik.tu-chemnitz.de}

\author{Bernd Hofmann}
\address{Department of Mathematics, Chemnitz University of Technology,
  09107 Chemnitz,  Germany}
\email{\tt bernd.hofmann@mathematik.tu-chemnitz.de}

\author{Ronny Ramlau}
\address{Johannes Kepler University, Industrial Mathematics Institute, Altenbergerstra{\ss}e 69, 4040 Linz, Austria}
\address{Johann Radon Institute for Computational and Appied Mathematics, Altenbergerstra{\ss}e 69, 4040 Linz, Austria}
\email{\tt ronny.ramlau@oeaw.ac.at}

\date{} %{June 28, 2013}



\thanks{$^*$Corresponding author.}

\begin{document}

\begin{abstract}
The convergence rates results in $\ell^1$-regularization when the sparsity assumption is narrowly missed, presented by Burger et al.~(2013 {\sl Inverse Problems} {\bf 29} 025013), are based on a crucial condition which requires that all basis elements belong to the range of the adjoint of the forward operator. Partly it was conjectured that such a condition is very restrictive.
In this context, we study sparsity-promoting varieties of Tikhonov regularization for linear ill-posed problems with respect to an orthonormal basis in a separable Hilbert space using $\ell^1$ and sublinear penalty terms. In particular, we show that the corresponding range condition is always satisfied for all basis elements if the problems are well-posed in a certain weaker topology and the basis elements are chosen appropriately related to an associated Gelfand triple.
The Radon transform, Symm's integral equation and linear integral operators of Volterra type are examples for such behaviour, which allows us
to apply convergence rates results for non-sparse solutions, and we further extend these results also to the case of non-convex $\ell^q$-regularization with $0<q<1$.
\end{abstract}

\maketitle

\section{Introduction}\label{s1}
\setcounter{equation}{0}
\setcounter{theorem}{0}

In recent years and first motivated by the seminal paper \cite{Daub03}, Tikhonov regularization for the stable approximate solution of ill-posed operator equations in Banach spaces based on $\ell^q$-norm-type penalties has been of considerable interest. This method is particularly suitable for situations where the unknown solution is sparse with respect to a given basis, because for $0 < q \leq 1$ the regularized solutions are themselves necessarily sparse and for $q=1$ convergence rates proportional to the noise level could be shown under standard source conditions in \cite{GrasHaltSch08}.
Regularization under sparsity constraints plays a prominent role for mathematical models in various fields like imaging (cf., e.g., \cite{Scherzetal09}) and parameter identification in partial differential equation problems (cf., e.g.,~\cite{JinMaass12}).
If the sparsity assumption fails, but the decay of the coefficients is fast enough to ensure $\ell^1$-solutions, then first convergence rates results for linear forward operators were developed in \cite{BFH13}, even in the absence of standard source conditions and approximate source conditions.
These results are valid under certain requirements on the interplay of the forward operator and the basis, the strongest of which is a sequence of range conditions assuming that all basis elements are in the range of the adjoint of the forward operator. At first glance this condition resembles properties of a singular value decomposition (SVD) and it is, indeed, satisfied if the forward operator maps compactly between two separable Hilbert spaces and the basis under consideration is precisely the orthonormal basis of eigenelements.

The purpose of this paper is twofold. First we establish that the aforementioned range conditions allow for a much larger class of operators and bases than just SVD-type and diagonal operator situations. To be precise, after introducing the used setting in Section~\ref{sec:preli} we show in Section~\ref{sec:gelfand} that the basis range conditions are always satisfied if the forward operator is continuously invertible with respect to a weaker topology -- defined in terms of a rigged Hilbert space and the corresponding Gelfand triple -- and the chosen basis for $\ell^q$-regularization is smooth enough.
Moreover, we give examples for such behaviour in Section~\ref{sec:gelfand}. Secondly, in Section \ref{sec:nonconvex} we turn to the case of non-convex regularization with $0<q<1$, where convergence rates are available from \cite{BreLor09, Grasm09} if the solution is sparse. However, if the sparsity assumption fails we derive new convergence rates results along the lines of \cite{BFH13}. Ideas for some 
extensions to frames and nonlinear operators are collected in Section~\ref{sec:extensions}.

\section{Preliminaries} \label{sec:preli}
\setcounter{equation}{0}
\setcounter{theorem}{0}

The focus of our study is on the treatment of specific \emph{ill-posed} operator equations
\begin{equation} \label{eq:opeq}
A x \,=\,y, \qquad x \in X ,\quad y \in Y,
\end{equation}
which serve as models for linear inverse problems arising in natural sciences, engineering, and finance. In this context, $A: X \to Y$ mapping between the infinite dimensional separable real Hilbert space $X$ and the Banach space $Y$ represents a \emph{bounded injective linear} forward operator with non-closed range  $\range(A)\not=\overline {\range(A)}^Y$. In the sequel, we denote by $\langle \cdot,\cdot\rangle_X$ and $\langle \cdot,\cdot\rangle_{Y^* \times Y}$ the inner product in $X$ and the duality pairing between $Y$ and its dual space $Y^*$, respectively, as well as by $\|\cdot\|_{X},\;\|\cdot\|_{Y}$ the corresponding norms in $X$ and $Y$. Taking into account that the
Hilbert space $X$ and its dual can be identified, the adjoint operator $A^*: Y^* \to X$ of $A$ is defined by the equation $\langle f, Ax \rangle_{Y^* \times Y} = \langle A^*f,x \rangle_X $ for all $x \in X$ and all $f \in Y^*$.
Since the Hilbert space $X$ is separable, there exist complete orthonormal systems (orthonormal bases) in $X$. Throughout the paper, we fix such an orthonormal basis $\{u_k\}_{k \in \N}$ of $X$, and we collect for any $x \in X$ the Fourier coefficients $\underline x_k=\langle x,u_k \rangle_X,\;k \in \N$, in an infinite sequence $\underline x:=\left(\underline x_1, \underline x_2,...\right)$.
By setting $\|\underline x\|_{q,w}:=\left( \sum \limits_{k=1}^\infty w_k |\underline x_k|^q\right)^{1/q}$ for exponents $0<q<\infty$ and with fixed weights $0 < w_0 \leq w_k$ for all $k \in \N$, which for $q \ge 1$ is the norm of the Banach space $\ell_w^q(\N)$, we introduce a scale of nested and with respect to $q$ monotonically increasing subsets
\begin{equation} \label{eq:scale}
\Mqw:=\{x \in X: \;\|\underline x\|_{q,w}<\infty \}, \qquad 0<q \le 2,
\end{equation}
of the Hilbert space $X$. Without going into detail, we mention that Besov spaces are of this particular form if the coefficients $\underline x$ are taken with respect to a wavelet basis and the weights are chosen appropriately (see, e.g. \cite{Mey}).
Since $X$ is isometrically isomorphic to $\ell^2(\N)$ with Parseval's identity $\|x\|_X=\|\underline x\|_2$, we have $\M_2=X$ as upper limit of the scale (\ref{eq:scale}), where as a rule we omit the index $w$ to indicate that $w_k = 1$ for all $k \in \N$.
 On the other hand, by introducing the support of $x \in X$ with respect to $\{u_k\}_{k \in \N}$ as $$\supp (x) := \{ k \in \N : \, \underline x_k \not= 0 \}$$ and $\,\|\underline x\|_0:= \sum \limits _{k=1}^\infty \sgn(|\underline x_k|)\,$
we find the lower limit of the scale (\ref{eq:scale}) as the set of \emph{sparse} elements, i.e.~of elements with \emph{finite support}  with respect to  $\{u_k\}_{k \in \N}$,
$$\M_0:=\{x \in X:\;\|\underline x\|_0<\infty\}.$$

The supposed specific character of the uniquely determined solution $\xdag \in X$ to (\ref{eq:opeq}) for $y \in \range(A)$  consists in the fact that either a sparsity assumption holds or that such an assumption is narrowly missed in the sense that $\langle \xdag,u_k\rangle_X$ decays to zero sufficiently fast as $k \to \infty$, or more precisely that $\xdag \in \Mqw$ for some $q>0$. This character motivates the use of $\ell^q_w$-regularization for computing stable approximate solutions to the ill-posed problem (\ref{eq:opeq}), where instead of $y$ only noisy data $y^\delta \in Y$ with $\|y-y^\delta\|_Y \le \delta$ for a given noise level $\delta>0$ are available. The regularized solutions $\xad \in X$ are minimizers of the Tikhonov functional $T_\alpha^\delta: X \to [0, \infty]$  defined as
\begin{equation} \label{eq:Tik}
T_\alpha^\delta(x):=\frac{1}{p} \|Ax-\yd\|_Y^p + \alpha \Rqw (x),
\end{equation}
where the penalty $\Rqw$ is defined as
\begin{equation} \label{eq:q}
 \Rqw (x) = \sum \limits _{k=1}^\infty w_k \left|\langle x,u_k \rangle_X \right|^q,  \qquad 0 < q \leq 1.
 \end{equation}

\subsection{Well-posedness and convergence}

For all $\yd \in Y$ as well as for all regularization parameters $\alpha>0$ minimizers $\xad \in X$ of the Tikhonov functional exist and are sparse, i.e.~$\xad \in \M_0$, and we refer for proof details to \cite{Grasm09} and
\cite[Prop.~4.5]{Gras10} for the case $0<q<1$ and to \cite[Lemma~2.1]{Lorenz08} for the case $q=1$. The existence of minimizers to $T_\alpha^\delta$ from (\ref{eq:Tik}) is based on the fact that the functional $\Rqw$ from (\ref{eq:q}) is weakly lower semi-continuous in $X$ and \emph{stabilizing},
which means that for all constants $c \ge 0$ the sublevel sets $\{x \in X: \Rqw(x) \le c\}$ are weakly sequentially compact in $X$. This stabilizing property follows from the weak coercivity of $\Rqw(x)$, i.e.~for all $x$ from the sublevel sets of $\Rqw$ the norm $\|x\|_X$ is uniformly bounded (cf.~\cite[Prop.~3.3]{Gras10}). Here, we also mention that  $\Row (x) \le \bigg( \frac{1}{w_0} \Rqw(x) \bigg)^{1/q}$ for all $0<q<1$.
In \cite{BreLor09,Grasm09,Gras10} one can find assertions on stability and convergence of such regularized solutions. If $\xdag \in \Mqw$ and the elements $x_n:=x_{\alpha_n}^{\delta_n}$ are regularized solutions for data $y^{\delta_n}\in Y$ satisfying $\lim \limits_{n \to \infty} \delta_n =0$, then we have convergence as $\lim \limits_{n \to \infty} \|x_n-\xdag\|_X =0$ whenever the regularization parameters $\alpha_n>0$ fulfil
the conditions $\lim \limits_{n \to \infty} \alpha_n=0$ and $\lim \limits_{n \to \infty} \frac{\delta^p_n}{\alpha_n}=0$. This norm convergence in $X$ follows from three ingredients which are all available:
  \begin{enumerate}
    \item the weak convergence of regularized solutions in $X$, which is a consequence of the weak compactness of the sublevel sets,
    \item the convergence $\lim \limits_{n \to \infty}\Rqw(x_n) =\Rqw(\xdag)$, which is a general property under the required behaviour of the regularization parameters, and
    \item the Kadec-Klee or Radon-Riesz property (cf.~\cite[Prop.~3.6]{Gras10}).
  \end{enumerate}

\subsection{Convergence rates}

Assuming sparsity, i.e.~if $\xdag \in \M_0$, numerous papers (cf., e.g.,~\cite{BreLor09,Lorenz08,Grasm09,Gras10}) even provide us with results on \emph{convergence rates}. A first convergence rate result when $\xdag \in \M_1$, but the sparsity assumption fails, was presented and proven in \cite{BFH13} and corresponding results for nonlinear forward operators were formulated in \cite{BoHo13}.

 We extend these results in Section~\ref{sec:nonconvex} to regularization with penalty term $\Rqw(x)$ also for exponents $0<q<1$. For appropriately chosen regularization parameters $\alpha_*=\alpha_*(\delta,y^\delta)$ we are interested in convergence rates of the form
\begin{equation} \label{eq:genrate}
E(x_{\alpha_*}^\delta,\xdag) = O(\varphi(\delta)) \qquad \mbox{as} \qquad \delta \to 0,
\end{equation}
with \emph{error measures} $E(\cdot,\cdot)$ and  concave \emph{index functions} $\varphi$, i.e. increasing functions $\varphi:~(0,\infty) \to (0,\infty)$ with $\lim \limits_{t \to +0} \varphi(t)=0$.
In agreement with previous works on convergence rates in $\ell^q$-regularization, we preferably consider the error measures
\begin{equation}\label{eq:Eq}
E(x,\xdag)=\Rqw (x-\xdag)=\|\underline x-\underline x^\dagger\|^q_{q,w}, \qquad \mbox{with } q \mbox{ from } (\ref{eq:Tik}),
\end{equation}
and
\begin{equation}\label{eq:E1}
E(x,\xdag)=\mathcal R_{1,w} (x-\xdag)=\|\underline x-\underline x^\dagger\|_{1,w}
\end{equation}
for arbitrary $0<q \le 1$ in the penalty term. If for such an error measure a \emph{variational inequality}
\begin{equation} \label{eq:vi}
 E(x,\xdag) \le \Rqw(x) - \Rqw (\xdag) + \varphi\left(\|A(x-\xdag)\| \right) \qquad \mbox{for all }  x \in \Mqw
\end{equation}
is valid with some concave index function $\varphi$, then we obtain a convergence rate (\ref{eq:genrate})
for the regularized solutions $x_{\alpha_*}^\delta$, for example by choosing the regularization parameter $\alpha_*=\alpha(\delta, \yd)$
 a-posteriori according to the
\emph{sequential discrepancy principle} (cf.~\cite{AnzHofMath12, HofMat12}). So far the discrepancy principle has mostly been studied in the context of convex regularization. However, in the recent publication \cite{AnzHofMath12} the most important results are carried over also to non-convex regularization, for example with the sublinear $\Rqw$ penalties which we consider here. For more details concerning the role of
variational inequalities for convergence rates in Tikhonov regularization we refer to \cite{HKPS07} and  \cite{AR2, Fle12, Flemmingbuch12, Scherzetal09, SchKHK12}.
Proposition~5.6 in \cite{BFH13} motivates, however, that the variational inequalities approach fails in $\ell^1$-regularization for non-injective forward operators. Therefore,
we also exclude non-injective operators $A$ in the present study.

The following assumption is the key ingredient for the derivation of convergence rates (\ref{eq:genrate}) in Section \ref{sec:nonconvex} (compare also \cite{BFH13, BoHo13}).  It combines smoothness of the basis elements $u_k$ expressed as a range inclusion with the smoothness of the solution $\xdag$ in terms of a rate of decay of its Fourier coefficients.

\begin{assumption} \label{ass:basic}
\begin{itemize} \item[]
\item[(a)] Suppose that the uniquely determined solution $\xdag \in X$ of equation (\ref{eq:opeq}) with $y \in \range(A)$ satisfies the condition $\xdag \in \Mqw$ for prescribed $0<q \le 1$.
\item[(b)] For all $k \in \supp (\xdag)$ there exist $f_k \in Y^*$ such that $u_k=A^*f_k$.
\end{itemize}
\end{assumption}

\smallskip

\begin{remark} \label{rem:smooth}
As is well-known any convergence rate in regularization of ill-posed operator equations in abstract spaces is connected with some kind of \emph{solution smoothness} (see, e.g. \cite{BakuKok04,EHN96,Scherzetal09,SchKHK12}), where \emph{source conditions} are the prominent form of expressing such smoothness. As one easily verifies, a source condition of type
\begin{equation} \label{eq:canon}
\xdag =A^*f, \qquad f \in Y^*,
\end{equation}
is an immediate consequence of Assumption~\ref{ass:basic} (b) whenever the solution is sparse, i.e.~if $\xdag \in \M_0$. Therefore basis range inclusions $u_k \in \range (A^*)$ play an important role in sparsity-promoting regularization, and we refer to the papers \cite{BreLor09, Grasm09, GrasHaltSch08, Lorenz08, Wangetal13}. In \cite[Sect.~4]{BFH13} it was discussed that source conditions and even approximate source conditions (cf.,~e.g.,~\cite{BoHo10,HeinHof09} and\cite[Part~III]{Flemmingbuch12}) always fail in $\ell^1$-regularization under Assumption~\ref{ass:basic} if the sparsity assumption is missed, i.e.~if $\xdag \notin \M_0$ and hence the support of $\xdag$ is an infinite
set of indices $k \in \N$. Nevertheless, for $q=1$ and $E$ from (\ref{eq:E1}) convergence rates (\ref{eq:genrate}) were derived for that case under Assumption~\ref{ass:basic}. Precisely, the two conditions $\xdag \in \M_{1,w}$
and Assumption~\ref{ass:basic} (b) represent the solution smoothness of $\xdag$ with respect to the forward operator $A$ in $\ell^1$-regularization of linear ill-posed operator equations, and their interplay determines the occurring convergence rates. We will see
in Section~4 that this is also the case in $\ell^q_w$-regularization, $0<q <1$, where the interplay of Items (a) and (b) yields convergence rates (\ref{eq:genrate}) with respect to $E$ from (\ref{eq:Eq}).
\end{remark}

\subsection{Chances and limitations}

As a consequence of the Closed Range Theorem (cf., e.g., \cite[p.~205]{Yos80}) the range $\range(A^*)$ is a non-closed subset of $X$, because $\range(A)$ is a non-closed subset of $Y$. Even though we have
$\overline{\range(A^*)}^X=X$ due to the injectivity of $A$, Assumption~\ref{ass:basic} (b) is a rather strong requirement, which together with the condition $\xdag \in \Mqw$ represents the
solution smoothness of $\xdag$ with respect to $A$ in $\ell^q_w$-regularization. Evidently, Item (b) can certainly be interpreted as an countably infinite set of smoothness conditions concerning the basis elements $u_k$.
In \cite[Remark 2.9]{BFH13} it was conjectured that all such conditions refer to a situation not too far from diagonal operators and the singular value decomposition. Really, in a Hilbert space setting
Item (b) is satisfied if $\{u_k\}_{k \in \N}$ acts as basis of eigenelements for $A$. If, however, the eigenbasis of $A$ occurs after rotating $\{u_k\}_{k \in \N}$ by means of a unitary operator $U: X \to X$, then
Item (b) will often get violated. A simple counterexample of an injective and compact linear operator $A$ with bidiagonal structure mapping in the separable Hilbert space $X$ with orthonormal basis $\{u_k\}_{k \in \N}$, where this item fails, was suggested by Markus Hegland (ANU, Canberra) as $Au_1:=u_1$ and $Au_k\:=\frac{u_k}{k}-\frac{u_{k-1}}{k-1},\; k=2,3,...$. Then we obviously have $u_1 \notin \range(A^*)$.
Roughly speaking, Assumption~\ref{ass:basic} (b)  will only hold if the cross connections between $\{u_k\}_{k \in \N}$ and the eigenbasis are not too turbulent. We will show in Section~3 that this is the case for a wide class of linear ill-posed problems, which are well-posed in a weaker topology, and for which then the convergence rates results of Section~4 are applicable whenever the basis elements  $u_k$ are smooth enough. It is future work to develop convergence rates if not all basis elements $u_k$ are smooth enough and hence Item (b) is violated, and we refer to \cite{FHH13} for preliminary results.

\section{Range inclusions, Gelfand triples, and examples} \label{sec:gelfand}
\setcounter{equation}{0}
\setcounter{theorem}{0}

\subsection{Range inclusions and Gelfand triples}

For the rate results from \cite{BFH13} and those which will be formulated in Section~\ref{sec:nonconvex}, the requirement $u_k \in \range(A^*),$ for all $k \in \N$, from Assumption~\ref{ass:basic} (b) is crucial.
So it is important to develop \emph{sufficient conditions} that allow us to construct a corresponding basis $\{u_k\}_{k \in \N}$ satisfying that point. In the simplest case one can directly characterize the subspace $\range(A^*)$ of the Hilbert spaces $X$ and choose a basis with suitable properties. Alternatively, under Assumption~\ref{ass:basic1} the following Propositions~\ref{pro:prop1} and \ref{pro:exist} provide us with an approach which is based on a range inclusion exploiting a space $V$ with stronger norm and continuously and densely embedded in $X$. As we will see below the premises of these propositions are fulfilled
for wide classes of practically relevant linear ill-posed operator equations. For example, smoothness along the lines of Assumption~\ref{ass:basic} (b) appears as a natural property if the forward operator $A$ is \emph{continuously invertible} under a weaker topology in the context of Gelfand triples whenever the basis elements $u_k$ are chosen in an adapted manner.

\begin{assumption} \label{ass:basic1}
  Suppose that the Hilbert space $X$ admits a separable linear subspace $V$ with norm $\|.\|_{V}$ which is dense in $X$, i.e.~$\overline{V}^X=X$, and such that
  \begin{enumerate}
   \item the corresponding linear embedding operator $\mathcal{E}: V \to X$ is continuous, i.e. with some constant $C>0$
      \begin{equation} \label{eq:emb}
	\|\mathcal{E} v\|_X \le C\,\|v\|_{V} \qquad \mbox{for all} \quad v \in V;
      \end{equation}
    \item the duality pairing between $V$ and its dual space $V^*$ is compatible with $\langle ., . \rangle_X$ in the sense that
      \begin{equation} \label{eq:comp}
	\langle x, v \rangle_{V^* \times V} = \langle x, v \rangle_X \qquad \mbox{for all} \quad x \in X,\; v \in V.
      \end{equation}
  \end{enumerate}
  \end{assumption}

\medskip

\begin{remark} \label{rem:Gelfand}
The pair $(V,X)$ of spaces satisfying the requirements of Assumption~\ref{ass:basic1} is called \emph{rigged Hilbert space}, and if we identify $X$ with its dual space $X^*$ by means of the Riesz isomorphism, the triple $(V,X,V^*)$ is called \emph{Gelfand triple}.
It is not difficult to construct Gelfand triples. Particular examples are $(\ell^q(\N), \ell^2(\N), \ell^{q'}(\N))$ and $(L^{q'} (\Omega), L^2(\Omega), L^q(\Omega))$ with $1<q<2$, $q'=q/(q-1)$ and bounded $\Omega \subset \R^n$ as well as $(H_0^s (\Omega), L^2 (\Omega), H^{-s} (\Omega))$ with $s>0$ and $\Omega \subset \R^n$ with sufficiently smooth boundary $\partial \Omega$. Observe that the compatibility condition (\ref{eq:comp}) ensures that the adjoint $\mathcal{E}^*$ of the embedding $\mathcal{E}$ is in turn the embedding of $X$ in $V^*$.
\end{remark}

We now turn to the main result of this section establishing a link between the range inclusion in Assumption~\ref{ass:basic} (b) and a smoothness condition $u_k \in V$ for all $k \in \supp (\xdag)$.  This is to say that all active elements of the chosen basis in $X$ must even belong to the smaller subspace $V$ possessing a stronger norm and turns out to be sufficient for Assumption~\ref{ass:basic} (b) to hold whenever $A$ is continuously invertible as an extension from $X$ to $V^*$. Thus, if the norm topology in $X$ is weakened to the norm topology in $V^*$ then the ill-posedness is lost.

\begin{proposition} \label{pro:prop1}
Under Assumption \ref{ass:basic1}, i.e.~if  $(V,X,V^*)$ is a Gelfand triple, assume that there is a constant $K>0$ such that
\begin{equation} \label{eq:K}
\|Ax\|_Y \ge K\,\|\mathcal{E}^*x\|_{V^*}  \qquad \mbox{for all}\qquad x \in X,
\end{equation}
then the range inclusion
\begin{equation} \label{eq:ri}
V \subseteq \range(A^*)
\end{equation}
 is valid. Hence, for any orthonormal basis $\{u_k\}_{k \in \N}$ in $X$ such that $u_k \in V$ for all $k \in \supp (\xdag)$ Assumption~\ref{ass:basic} (b) is satisfied.
\end{proposition}

In order to prove Proposition~\ref{pro:prop1} we exploit the following technical lemma which was formulated and proven for general normed linear spaces as Lemma~8.21 in \cite{Scherzetal09}.

\begin{lemma} \label{lem:aux}
Let $A: X \to Y$ denote a bounded linear operator mapping between the real Hilbert spaces $X$ and $Y$ and let $u \in X$. Then the conditions $u \in \range(A^*)$ and
  \begin{equation} \label{eq:topr}
    \vert\langle x,u\rangle_X \vert \le C_u\,\|Ax\|_Y
  \end{equation}
  for some constant $C_u>0$ and for all $x \in X$ are equivalent.
\end{lemma}

\noindent {\bf Proof of Proposition~\ref{pro:prop1}.}
  For arbitrary $v \in V$ it follows from (\ref{eq:K}) that
    $$\vert\langle x, v \rangle_X\vert = \vert\langle x, \mathcal{E} v \rangle_X\vert=\vert\langle \mathcal{E}^* x,v\rangle_{V^* \times V}\vert \le \|v\|_{V}\|\mathcal{E}^* x\|_{V^*}\le \frac{\|v\|_{V}}{K}\,\|Ax\|_Y $$
  holds for all $x \in X$. This is (\ref{eq:topr}) with $C_v= \frac{\|v\|_{V}}{K}$, and Lemma~\ref{lem:aux} yields the range inclusion (\ref{eq:ri}).
\hfill\fbox{}

\bigskip

It is important to mention that Proposition \ref{pro:prop1} applies for all problems where the forward operator $A: X \to Y$ fails to be continuously invertible, but in terms of a Gelfand triple $(V,X,V^*)$ can be extended to $V^*$ such that $A: V^* \to Y$ is a continuous linear isomorphism from $V^*$ onto $Y$. As the subsequent series of examples will show, such isomorphisms occur for wide classes of problems. Precisely, we can make use of inequality chains of the form
  \begin{equation} \label{eq:isom}
      \underline c \| x \|_{V^*} \leq \| A x \|_Y \leq \overline c \| x \|_{V^*} \qquad \mbox{for all} \quad x \in X
    \end{equation}
with constants $0<\underline c \le \overline c<\infty$, where the left hand side inequality yields the required estimate (\ref{eq:K}) with $K:=\underline c$.

\medskip

  To clarify the implications of the above considerations for convergence rates we roughly summarize some facts which will be proven in extended form in Corollary~\ref{cor:sdp} below. In particular, the explicit structure of the rate function $\varphi$ as well as an appropriate choice of the regularization parameter will be supplied ibidem.

  \begin{theorem} \label{thm:summary}
  Suppose that $(V,X,V^*)$ is a Gelfand triple and that an inequality (\ref{eq:K}) holds for the forward operator $A$. If the orthonormal basis $\{u_k\}_{k \in \N}$ in $X$ satisfies $u_k \in V$ for all $k \in \supp (\xdag)$, the solution satisfies $\xdag \in \M_1$ and if the regularization parameter $\alpha > 0$ is chosen appropriately, then there is a concave index function $\varphi$ such that
    $$\| \underline x_{\alpha}^\delta- \underline x^\dag\|_{\1} = \mathcal{O}(\varphi(\delta)) \qquad \mbox{as} \qquad \delta \to 0.$$
  \end{theorem}

\medskip

The question remains whether a basis in the Hilbert space $X$, consisting of elements in a dense separable subspace $V$, necessarily exists. A positive and constructive answer can be found by use of the Gram-Schmidt process of orthonormalization and we will mention further examples in the following subsections.

\begin{proposition} \label{pro:exist}
 Let  $(V,X,V^*)$ be a Gelfand triple, then there exists an orthonormal basis $\{u_k\}_{k \in \N}$ in $X$ such that $u_k \in V$ for all $k \in \N$.
\end{proposition}

\begin{proof}
Since $V$ is separable there exists a countable and dense subset $\{ v_k\}_{k \in \N}$ in $V$ and as a consequence of (\ref{eq:emb}) we even have $\,\overline{\{ v_k\}_{k \in \N}}^{X}=X$. Hence, we can construct an orthonormal basis $\{u_k\}_{k \in \N} \subset V$ in $X$ by the Gram-Schmidt process of orthonormalization (discarding all zeros that occur in the course of the algorithm).
This completes the proof.
\end{proof}

\subsection{Reconstructions based on the Radon transform}

A prototypical example of a linear ill-posed problem is the interpretation of indirect measurements based on the Radon transform as the Radon transform represents the mathematical forward model of computerized tomography (CT). It maps a function $x(t)$ defined on the unit disc $\Omega \subset \R^2$ onto the family of line integrals
    \[  \RT x (s) = \int_{I(\theta, s)} x(t) dt, \qquad \theta \in S^1, \quad -1 \leq s \leq 1, \]
where $S^1 \subset \R^2$ denotes the unit sphere and $I(\theta, s) = \{ t \in \Omega  : t \cdot \theta = s \}$ the line segment in $\Omega$ perpendicular to $\theta$ with distance $|t|$ from the origin. For a detailed introduction to the Radon transform we refer the interested reader to \cite{n86, nw01}.

For our purpose, it is important that the Radon transform is continuously invertible between certain Sobolev spaces. We consider the real Hilbert spaces $X=L^2(\Omega)$ and $Y = L^2(S^1 \times [-1,1])$ and recall Theorem 2.10 from \cite{nw01}
formulated as Proposition~\ref{pro:rtcont}, which yields a variety of the inequality chain (\ref{eq:isom}).

\begin{proposition} \label{pro:rtcont}
 There exist constants $0<\underline c \le \overline c<\infty$ such that
    \begin{equation} \label{eq:rtcont2}
      \underline c \| x \|_{H^{-1/2} (\Omega)} \leq \| \RT x \|_Y \leq \overline c \| x \|_{H^{-1/2} (\Omega)} \qquad \mbox{for all} \quad x \in H^{-1/2} (\Omega).
    \end{equation}
\end{proposition}

\medskip

If we consider the forward operator $A$ in (\ref{eq:opeq}) as a composition $A := \RT \, \circ \,\mathcal E^* : X \to Y$, where $\mathcal E^*: L^2 (\Omega) \to H^{-1/2} (\Omega)$ is the Sobolev embedding operator,
then we have an estimate (\ref{eq:K}) with $V^*:=H^{-1/2}(\Omega)$ and $K:=\underline c$. Observe that $A$ is compact and injective and that $\big(V, X, V^*)$ with $V:=H_0^{1/2}(\Omega)$ is a Gelfand triple. Consequently, Proposition~\ref{pro:prop1} is applicable and yields the range inclusion
 (\ref{eq:ri}). Hence Item (b) of Assumption~\ref{ass:basic} is fulfilled if the chosen basis $\{u_k\}_{k \in \N}$ in $L^2(\Omega)$ is contained in $H_0^{1/2} (\Omega)$.

Clearly this requirement is met if $\{u_k\}_{k \in \N}$ belongs even to $H_0^{s} (\Omega)$ with $s > 1/2$. In particular,
the eigenfunctions of the Laplace operator with homogeneous Dirichlet boundary conditions form an orthonormal basis of $L^2(\Omega)$ consisting of elements in $H^1_0 (\Omega)$. They are explicitely given \cite{GreNg13} in polar coordinates $(r, \varphi) \in [0,1] \times [0, 2\pi)$ and in terms of the Bessel-functions $J_n (r)$ by
    \begin{equation} \label{eq:Bessel}
      u_{n,k,l} (r, \varphi) = c_{n,k} ~ J_n (\beta_{n,k} r) \cdot \Bigg\{ \begin{array}{cl} \cos (n \varphi) & l=1\\ \sin (n \varphi) & l=2, n \neq 0, \end{array}
    \end{equation}
  with $n$, $k=0,1,2, \dots$, $l=1,2$, where $\beta_{n,k}$ denotes the $k$-th root of $J_n (r)$ and $c_{n,k}$ the normalizing constant $c_{n,k} = \sqrt{2/\pi} ~ \left| J_{n+1} (\beta_{n,k}) \right|^{-1}$.

\subsection{Solving Symm's integral equation}

The following example is taken from \cite[Sect.~3.3]{Kirsch11} (see also \cite[Sect.~6.1.2.1]{Rieder03}). In potential theory the solution $z$ of the Dirichlet problem for the Laplace equation
$$\Delta z=0 \quad \mbox{in} \quad \Omega\,, \qquad \quad z=\widetilde y  \quad \mbox{on} \quad\partial\Omega\,,$$
where $\Omega \subset \R^2$ is a bounded and simply connected domain with analytic boundary $\partial \Omega$ and $\widetilde y \in C(\partial \Omega)$ is a given function,
can be found by the simple layer potential
$$z(\kappa)=\frac{1}{\pi} \int \limits_{\partial \Omega} \widetilde x(\kappa) \log |\kappa-\eta|\, d s(\eta), \qquad \kappa \in \Omega\,.$$
In this context, the required potential density function $\tilde x \in C(\partial \Omega)$ satisfies \emph{Symm's integral equation}
\begin{equation} \label{eq:symm}
-\frac{1}{\pi} \int \limits_{\partial \Omega} \widetilde x(\kappa) \log |\kappa-\eta|\, d s(\eta)=\widetilde y(\eta), \qquad \eta \in \partial\Omega\,.
\end{equation}
Assuming that the boundary $\partial \Omega$ can be parametrized by a smooth curve $\gamma: [0,2\pi] \to \R^2$ with non-vanishing derivative, the analog of equation (\ref{eq:symm}) for the transformed density
$x(t):=\tilde x(\gamma(t))\,|\dot{\gamma}(t)|$  attains the form
\begin{equation} \label{eq:symmpar}
[Ax](t):= -\frac{1}{\pi} \int \limits_0^{2\pi} x(t)\, \log |\gamma(t)-\gamma(\tau)|\, d \tau= y(t), \qquad t \in [0,2\pi]\,.
\end{equation}
We denote by $L^2_{per}(0,2\pi)$ the complex Hilbert space of square integrable periodic functions with period length $2\pi$ and
by
 $$H^\nu_{per}(0,2\pi):= \Bigg\{ \sum \limits_{k \in \Z}c_k e^{ikt}:\,\sum \limits_{k \in \Z}|c_k|^2(1+k^2)^{\nu}<\infty \Bigg\}$$
the corresponding complex Sobolev space of level $\nu \in \R$.  Then with $X=Y:=L^2_{per}(0,2\pi)$ the linear operator $A: X \to Y$  defined by (\ref{eq:symmpar}) is a compact map with nonclosed range $\range(A) \subset Y$, but we have a continuous linear isomorphism (\ref{eq:isom}) with $V^*:=H^{-1}_{per}(0,2\pi)$. To apply
Proposition~\ref{pro:prop1} it is sufficient to have a countable set in $H^{1}_{per}(0,2\pi)$ which is an orthonormal basis in  $L^2_{per}(0,2\pi)$. Evidently, the sequence $\{e^{ikt}/\sqrt{2\pi}\}_{k \in \Z}$ of analytic functions serves as such a system.

\subsection{Solving Abel integral equations of first kind}

A fairly general approach to the premises of Proposition \ref{pro:prop1} is delivered by \emph{Hilbert scales} $\{X_\nu\}_{\nu \in \R}$
(cf., e.g.,~\cite[Chapt.~5]{Bau87}).
The Hilbert scales are generated on the basis of a positive definite self-adjoint unbounded operator $B:\domain{B}\subset X \to X$ densely defined in $X$ with discrete spectrum completely characterized by eigenvalues $0<\mu_1 \le \mu_2 \le ... \le \mu_k \le \mu_{k+1} \le ... \to \infty$ as $k \to \infty$ and a orthonormal basis $\{w_k\}_{k \in \N}$ of corresponding
eigenelements satisfying the equation $Bw_k=\mu_kw_k,\;k \in \N$. Then for every $\nu \in \R$ we can define an inner product $\langle u,v \rangle_{X_\nu}:=\mu_k^{2\nu}\langle u,w_k\rangle_X  \langle v,w_k \rangle_X$ and an
associated norm $\|u\|_{X_\nu}:=\sqrt{\langle u,u\rangle_{X_\nu}}$ and we have $x \in X_\nu$ for some $x \in X$ if and only if $\|x\|_{X_\nu}<\infty$. Taking into account that $X=X_0$ we have for every $\nu>0$ and
 $V:=X_\nu,\;V^*=X_{-\nu})$ a Gelfand triple $(V,X,V^*)$ as introduced above.

We present such an example with the particularity that the Hilbert scale is related to an associated Sobolev scale. Precisely, our focus is on a class of operator equations (\ref{eq:opeq}) formed by Abel-type integral equations of the first kind, and we refer to \cite{GorVes91} for more details and practical applications. Let $\nu$ be a fixed parameter chosen from the interval $0<\nu \le 1$. Then we consider in the real Hilbert space $X=Y=L^2(0,1)$ the weakly singular  integral operator $A$ defined as
\begin{equation} \label{eq:abel}
[Ax](s) = \frac{1}{\Gamma(\nu)} \int \limits _0^s (s-t)^{\nu-1}\, K(s,t)\, x(t)\, dt,\qquad 0 \le s \le 1,
\end{equation}
where the kernel function $K(s,t)$ is assumed to be continuous for $0 \le t \le s \le 1$, with $K(t,t)=1,\;0 \le t \le 1$ and a decreasing function $\kappa \in L^2(0,1)$ such that $\left|\frac{\partial K}{\partial t}(s,t)\right| \le \kappa(t),\;0 \le t \le s \le 1$.

Now the differential operator $B$ for the present example is mapping in $X=L^2(0,1)$ and defined as $B^2w=w^{\prime\prime}$ with the boundary conditions $w^\prime(0)=w(1)=0$. It possesses the eigenvalues
$\mu_k=(k-\frac{1}{2})\pi$ and the corresponding eigenfunctions $w_k(t)=\sqrt{2} \cos (\mu_k t),\;0 \le t \le 1$. Moreover, the Hilbert scale elements for $0<\nu \le 1$ are closely connected with Sobolev spaces, namely
$$X_\nu=\left\{\begin{array}{clc} H^\nu(0,1)& \mbox{for} & 0<\nu<1/2\\\{w \in H^{1/2}(0,1):\,\int \limits_0^1 (1-t)^{-1}|w(t)|^2 dt<\infty\}& \mbox{for} & \nu=1/2\\\{w \in H^\nu(0,1):\,w(1)=0\} & \mbox{for} & 1/2<\nu \le 1. \end{array}\right.$$
Following the studies in \cite{GorYam99} it follows that there are constants $0<\underline c \le \overline c < \infty$ such that in $X=L^2(0,1)$ for the Abel-type operator $A$ from (\ref{eq:abel}) an inequality chain of the form (\ref{eq:isom}) is valid as
$$\underline c\,\|x\|_{X_{-\nu}} \le \|Ax\|_X  \le \overline c \,\|x\|_{X_{-\nu}}$$
for all $x \in X$. Hence for basis functions $u_k \in X_\nu,\;k \in \N$, the Proposition~\ref{pro:prop1} applies and yields the range inclusion (\ref{eq:ri}) as required.
Because of $X_1 \subset X_\nu$ for all $0<\nu \le 1$ the eigenbasis  $\big \{\sqrt{2} \cos((k-\frac{1}{2})\pi t),\;0 \le t \le 1 \big\}_{k \in \N} \subset X_1$ of $B$ acts as such an orthonormal basis $\{u_k\}_{k \in \N}$ in $L^2(0,1)$.

\subsection{Finding higher derivatives from noisy data}

Considering the real Hilbert space $L^2(0,1)$ as $X$ and $Y$, the constant kernel $K(s,t)=1$ and $\nu=n$ with $n=2,3,...$, the operator $A$ from (\ref{eq:abel}) leads to the problem of finding the $n$-th derivative of a function $y$ when the equation (\ref{eq:opeq}) with forward operator
 \begin{equation} \label{eq:nthder}
[Ax](s) = \frac{1}{\Gamma(n)} \int \limits _0^s (s-t)^{n-1}\, x(t)\, dt,\qquad 0 \le s \le 1,
\end{equation}
is to be solved.  Since additional boundary conditions occur for integers $n$ greater than~$1$, the characterization
of $X_n$ and $X_{-n}$ gets more complicated than for the parameter interval $0<\nu \le 1$ in Subsection~3.3 and the Gelfand triples are not so simple. However, sufficient conditions for Item (b) of Assumption~\ref{ass:basic} can be derived directly without the detour via Proposition~\ref{pro:prop1} when the range $\range(A^*)$ can be verified explicitly. It is well known that such \emph{fractional integral operators} have the adjoint
$$[A^*z](t) = \frac{1}{\Gamma(n)} \int \limits_t^1  (s-t)^{n-1}\, z(s)\, ds,\qquad 0 \le t \le 1.$$
This explicit structure of $A^*$ allows for a characterization of its range. So we have for example in the case $n=2$
$$\range(A^*)=\{w \in H^2(0,1):\, w(1)=0,\; w^\prime(0)=0\},$$
which determines the differentiability requirements and boundary conditions to be prescribed for all basis functions $u_k, \; k\in \N$. Note that again
  $$\{u_k\}_{k \in \N}= \Bigg \{\sqrt{2} \cos \Big( \bigg(k-\frac{1}{2}\bigg)\pi t \Big),\;0 \le t \le 1 \Big\}_{k \in \N}$$
serves an appropriate basis in this case.

\section{Non-convex sparsity-promoting regularization}\label{sec:nonconvex}
\setcounter{equation}{0}
\setcounter{theorem}{0}

  Although for $0<q<1$ the functional $\Rqw (x)$ is not convex, which leads to additional difficulties in the numerical computation of regularized solutions $\xad$ (see~\cite{RamZar12} for a minimization approach), Tikhonov regularization with such sublinear penalties is of practical importance, because the regularized solutions $\xad$ are sparse and such penalties emphasize  (cf.~\cite{Zar09,Xuetal10}) the sparse character of the solutions even more than convex regularization with $q = 1$. In various biomedical applications, for example, the principle of Occam's razor suggests that solutions composed of few active elements should be preferred to explain certain observations. %as long as they result in the same 
  We also refer to \cite{KMS12} for a sensitivity discussion of the case where $q$ is slightly greater than~$1$, but the regularized solution are not necessarily sparse. Nevertheless, it may happen that $\xdag \in \M_0$ is expected, but this property holds true only approximately.
  Therefore, it seems reasonable to extend convergence rates results from \cite{BFH13} obtained for $\xdag \in \M_{1,w}$, to the non-convex penalty case under the assumption $\xdag \in \Mqw,\;0<q<1$. We note that convergence rates for $\xdag \in \M_0$ and $0<q<1$ were presented in \cite{BreLor09,Grasm09,Gras10,Grasm10,Zar09} and also that the limit case of using $q=0$ with $\,\Rzw (x)= \sum \limits _{k=1}^\infty \sgn(|\langle x,u_k \rangle_X |)\,$  makes no sense in a pure form, because this penalty is not stabilizing and hence minimizers of (\ref{eq:Tik}) might not exist (cf.~\cite[\S~5.2]{Lorenz08}).
  However, these difficulties can be reduced if $\Rzw (x)$ is used in a linear combination with $\Rtw (x)$ as suggested in \cite{Wangetal13}. We also note that studies for penalties of type $\,\Reg(x)= \sum \limits _{k=1}^\infty \psi(|\langle x,u_k \rangle_X |)\,$   taking into account a wider class of nonnegative functions $\psi$ were presented in \cite{BreLor09, Gras10}.

  We recall the definition of the sequential discrepancy principle from \cite{AnzHofMath12, HofMat12}. For prescribed $0 < \theta < 1$ and $\alpha_0>0$, let
    \[ \Delta_\theta := \{ \alpha_j ~: ~ \alpha_j = \theta^j \alpha_0, \quad j \in \Z \}. \]
  Given any $\delta >0$  and data $\yd$, we fix some selection of minimizers~$\xad$ of the Tikhonov functional $T_\alpha^\delta$ for $\alpha \in \Delta_{\theta}$.

\begin{definition}[Sequential discrepancy principle] \label{dfn:sdp}
  We say that an element~$\alpha \in \Delta_\theta$ is
  chosen according to the {\it sequential discrepancy principle (SDP)}, if
  \begin{equation}\label{eq:sdp}
    \| A \xad - \yd \| \leq \tau \delta < \| A {x_{\alpha/\theta}^{\delta}} - \yd \|.
  \end{equation}
\end{definition}

  The main advantage of this sequential definition over stronger versions of the discrepancy principle, in particular
    \begin{equation}\label{eq:strongdp}
      \tau_1 \delta \leq \| A \xad - \yd \| \leq \tau_2 \delta, \qquad \mbox{with } 1 \leq \tau_1 \leq \tau_2
    \end{equation}
  (cf. \cite{AR1, AR2, Bon09, BFH13, Flemmingbuch12}), becomes evident when the forward operator under consideration is non-linear and (\ref{eq:strongdp}) may not have any solutions due to duality gaps (cf. \cite{SchKHK12}). But even for linear $A$, we feel that Definition \ref{dfn:sdp} captures the computational reality of finding such regularization parameters $\alpha$ more accurately and we thus prefer to work with the sequential formulation here. However, it is certainly worth noting, that the following results remain true if $\alpha$ is chosen according to (\ref{eq:strongdp}) instead.

\medskip

  To guarantee regularizing properties of Tikhonov regularization with parameter $\alpha$ chosen according to the discrepancy principle, it is crucial that the so-called {\em exact penalization veto} is satisfied and the data are {\em compatible} (compare Assumptions 2 and 3 in \cite{AnzHofMath12}, respectively).
  In the context of convex regularization with $\Rqw (x)$ from (\ref{eq:q}), $1 \le q \le 2$, and bounded, linear and injective forward operator $A$ these assumptions are necessarily satisfied if $p>1$ in (\ref{eq:Tik}) -- which excludes exact penalization -- and if the solution satisfies $\xdag \neq 0$.
  The same remains true for the non-convex penalties $\Rqw (x)$ from (\ref{eq:q}) with $0<q<1$ as they are positively homogeneous of degree $q$. The proof, which is based on arguments from \cite{AnzHofMath12, AR1}, is included here for the convenience of the reader.

  \begin{theorem} \label{thm:sdp}
    In the Tikhonov functional (\ref{eq:Tik}), let $p>1$ and $0<q<1$ such that Assumption \ref{ass:basic} (a) holds for $\xdag \neq 0$. Then there is some $\bar \delta>0$ such that regularization parameters $\alpha = \alpha(\delta, \yd)$ chosen according to the SDP exist for all $0<\delta \le \bar \delta$. These parameters satisfy the limit conditions
    \begin{equation} \label{eq:limits}
     \alpha (\delta, \yd) \to 0 \qquad \mbox{and} \qquad \frac{\delta^p}{\alpha (\delta, \yd)} \to 0 \qquad \mbox{as} \quad \delta \to 0.
    \end{equation}
    and the associated regularized solutions $x_{\alpha (\delta, \yd)}^\delta$ converge to $\xdag$ in norm and in the sense $\lim \limits_{\delta \to 0} \Rqw (x_{\alpha (\delta, \yd)}^\delta - \xdag) = 0$ as $\delta \to 0$.
  \end{theorem}

  \begin{proof}
   We use some notation from \cite{AnzHofMath12}: The unique minimizing element of the penalty term $\Rqw (x)$ is $x_{\min} =0$, so that $X_{\min} = Y_{\min} = \{ 0 \}$, $\Reg_{\min} = 0$. Moreover, $\mathcal{L} := \{ \xdag \}$ yields $\mathcal{L} \cap X_{\min} = \emptyset$. Thus, according to Proposition 5 in \cite{AnzHofMath12}, both the data compatibility assumption and the exact penalization veto are satisfied if the solution $\xdag$ does not minimize $T^0_\alpha$ for any $\alpha > 0$, where by $T^0_\alpha$ we denote the Tikhonov functional (\ref{eq:Tik}) with exact data $y$. To prove the latter assertion, we proceed by contradiction:

    Assume that $\xdag \neq 0$ minimizes $T^0_\alpha$ for some $\alpha > 0$, then for all $0 < t < 1$ we have
      \begin{equation*}
	\alpha \Rqw (\xdag) = T^0_\alpha (\xdag) \leq T^0_\alpha ((1-t) \xdag) = t^p \| y \|^p + \alpha (1-t)^q \Rqw(\xdag),
      \end{equation*}
    which yields
      \begin{equation*}
	\alpha \Rqw (\xdag) \leq \frac{t^p}{1-(1-t)^q}  \| y \|^p \to 0 \qquad \mbox{as }\quad  t \to +0.
      \end{equation*}
    It follows that $\Rqw (\xdag)=0$ and hence $\xdag = 0$ which contradicts our assumption.

    Consequently, Lemma 2 and Theorem 1 in \cite{AnzHofMath12} are applicable and yield the existence of a regularization parameter $\alpha = \alpha (\delta, \yd) \in \Delta_\theta$ chosen according to (\ref{eq:sdp}), the asymptotic relations (\ref{eq:limits}) and, since the solution $\xdag$ is unique, also weak convergence $\xad \rightharpoonup \xdag$ as well as $\Rqw (\xad) \to \Rqw (\xdag)$ as $\delta \to 0$.

    Finally, the functionals $\Rqw (x)$ satisfy the Kadec-Klee (or Radon-Riesz) property which implies $\Rqw (\xad - \xdag) \to 0$ (see \cite[Prop.~3.6]{Gras10})
    and thus also $\| \xad - \xdag \|_X \to 0$ as $\delta \to 0$.
  \end{proof}

  To go beyond well-posedness and convergence of the regularized solutions, we use the smoothness of basis elements $u_k$, expressed in terms of the range inclusion $u_k \in \range (A^*)$. Indeed, under our specific Assumption~\ref{ass:basic} we have an adapted variational inequality (\ref{eq:vi}) for the error measure $E(x,\xdag)=\Rqw (x-\xdag)$ which immediately leads to associated convergence rates.

  \begin{theorem} \label{thm:vi}
  Under Assumption~\ref{ass:basic} and taking any $0<q\le 1$ a variational inequality
    \begin{equation} \label{eq:viq}
      \Rqw (x-\xdag) \leq \Rqw (x) - \Rqw (\xdag) + \varphi_{q,w} \left(\| A (x- \xdag) \| \right)
    \end{equation}
  holds for all $x \in \Mqw$ with the concave index function
    \begin{equation} \label{eq:phiq}
      \varphi_{q,w} (t) := 2 \inf_{n \in \N} \left( \sum_{k = n+1}^\infty w_k \vert \langle \xdag,u_k\rangle_X \vert^q + t^q \sum_{k \in S_n (\xdag)} w_k \|f_k\|_{Y^*}^q \right),
    \end{equation}
   where
    \[ S_n (\xdag) = \supp (\xdag) \cap \{1, \ldots, n \}. \]
  \end{theorem}

\begin{proof}
  We sketch the proof, which is an extension of arguments from Lemma 5.1 and Theorem 5.2 in \cite{BFH13}. Let us denote by $L^* : X \to \ell^2 (\N)$ the operator mapping elements $x \in X$ to their coefficients $\{ \langle x, u_k \rangle \}_{k \in \N}$ (compare Section \ref{sec:frame}) and define projections $P_n : \ell^2 (\N) \to \ell^2 (\N)$ according to
    \[ P_n (c)_k = \bigg\{\begin{array}{cl} c_k & \mbox{if}\; k \in S_n (\xdag)   \\
		0 & \mbox{otherwise.} \end{array} \]
  Using the identity
    \[ \Rqw (x) = \| P_n L^* x \|_{q,w}^q + \left\| (I-P_n) L^* x \right\|_{q,w}^q, \qquad x \in \Mqw, \]
  and the $q-$triangle inequality for $0<q \leq 1$,
    \[ \Rqw (x + z) \leq \Rqw (x) + \Rqw (z), \qquad x, z \in \Mqw, \]
  we derive
    \[  \Rqw (x - \xdag) - \Rqw (x) + \Rqw (\xdag) \leq 2 \left( \left\| (I - P_n) L^* \xdag \right\|_{q,w}^q + \left\|  P_n L^*( x- \xdag) \right\|_{q,w}^q \right). \]
  Concerning the latter term, recall that for $k \in S_n (\xdag) \subset \supp (\xdag)$ there exists $f_k \in Y^*$ such that $u_k = A^* f_k$ and thus
    \[  \left\|  P_n L^*( x- \xdag) \right\|_{q,w}^q = \sum_{k \in S_n (\xdag)} w_k |\langle x-\xdag, A^* f_k \rangle_X|^q \leq \| A(x-\xdag) \|_Y^q \sum_{k \in S_n (\xdag)} w_k \|f_k\|_{Y^*}^q,  \]
  whence the result (\ref{eq:viq}) follows.

  To see that the continuous function $\varphi_{q,w} (t)$ is indeed a concave index function, observe that $\varphi_{q,w} (t) < \infty$ for all $t \in [0, \infty)$ and that $\varphi_{q,w} (0) = 0$. However, it is well-known that an infimum of continuous and concave functions $\varphi_{q,w}$ is itself concave. For $0 \leq t_1 < t_2 < \infty$ it holds for all $n \in \N$ that
    \[  \sum_{k = n+1}^\infty w_k \vert  \langle \xdag,u_k\rangle_X \vert^q + t_1^q \sum_{k \in S_n (\xdag)} w_k \|f_k\|_{Y^*}^q \leq \sum_{k = n+1}^\infty w_k \vert  \langle \xdag,u_k\rangle_X \vert^q + t_2^q \sum_{k \in S_n (\xdag)} w_k \|f_k\|_{Y^*}^q, \]
  hence $\varphi_{q,w}$ is non-decreasing.
\end{proof}

  For regularization parameter chosen according to the SDP, it has been shown in \cite[Theorem 2]{HofMat12} that the presence of a variational inequality (\ref{eq:viq}) yields a corresponding convergence rate as the noise level $\delta$ tends to zero.

\begin{corollary} \label{cor:cr}
  Under the assumptions of Theorems~\ref{thm:sdp} and \ref{thm:vi} and using SDP as regularization parameter choice $\alpha_*=\alpha_*(\delta,y^\delta)$ we obtain a convergence rate
    \begin{equation} \label{eq:rate}
      \Rqw (x^\delta_{\alpha_*}- \xdag) = \mathcal{O} (\varphi_{q,w} (\delta)) \qquad \mbox{as} \qquad \delta \to 0
    \end{equation}
 with the concave index function $\varphi_{q,w}$ from (\ref{eq:phiq}).
  Independent of the choice of $0<q \le 1$ in the penalty we always have
      \[ \| x^\delta_{\alpha_*}- \xdag \|_X \leq \Row (x^\delta_{\alpha_*}- \xdag) %=  \|\underline{x}^\delta_{\alpha_*} -   \xdags\|_{1,w}
	  \le \Big( \frac{1}{w_0} \Rqw (x^\delta_{\alpha_*}- \xdag) \Big)^{1/q} =  \mathcal{O} (\varphi^{1/q}_{q,w} (\delta)) \]
  as $\delta \to 0$.
   If $\xdag$ is sparse, then $\xdag \in \Mqw$ holds for all $q \in [0,1]$ and we get a linear rate expressed by the estimate
   \[ \varphi_{q,w}^{1/q} (\delta) \le C_{q,w} \,\delta, \]
  but the constants
    \[ C_{q,w} = \left( \sum_{k \in \supp (\xdag)} w_k \|f_k\|_{Y^*}^q \right)^{1/q} \]
  tend to infinity as $q \to 0$ whenever $~\supp (\xdag)$ contains more than one element.
\end{corollary}

  \begin{remark}  \label{rem:nonlin}
	Theorem \ref{thm:vi} as well as Corollary \ref{cor:cr} remain applicable also for nonlinear forward operators $F(x)$ satisfying the assumptions outlined in Section \ref{sec:nonlin} (compare \cite{BoHo13}). To update the variational inequality (\ref{eq:viq}) we only alter the estimate
	\begin{eqnarray*}
	   \left\|  P_n L^*( x- \xdag) \right\|_{q,w}^q & %= \sum_{k \in S_n (\xdag)} w_k |\langle x-\xdag, A^* f_k \rangle_X|^q
	   			 \leq \| F'(\xdag) (x-\xdag) \|_Y^q \sum_{k \in S_n (\xdag)} w_k \|f_k\|_{Y^*}^q\\
	   			 & \leq \sigma (\| F(x)-F(\xdag) \|_Y)^q \sum_{k \in S_n (\xdag)} w_k \|f_k\|_{Y^*}^q,
	\end{eqnarray*}
	which yields
    \begin{equation*}
      \Rqw (x-\xdag) \leq \Rqw (x) - \Rqw (\xdag) + \varphi_{q,w,\sigma} \left(\| F(x)- F(\xdag) \| \right)
    \end{equation*}
     for all $x \in M$, where
    \[ \varphi_{q,w,\sigma} (t) = 2 \inf_{n \in \N} \left( \sum_{k = n+1}^\infty w_k \vert \langle \xdag,u_k\rangle_X \vert^q + \sigma(t)^q \sum_{k \in S_n (\xdag)} w_k \|f_k\|_{Y^*}^q \right) \]
    is the new index function and determines the rate of convergence in (\ref{eq:rate}).
  \end{remark}

   We also mention a lower bound for the asymptotic behaviour of the SDP from \cite{HofMat12} that follows from the variational inequality (\ref{eq:viq}).

  \begin{corollary} \label{cor:sdp}
    Under the assumptions of Theorem~\ref{thm:sdp} and \ref{thm:vi} the regularization parameter $\alpha_*(\delta,y^\delta) \in \Delta_\theta$ chosen according to the SDP satisfies
      \[ \alpha_*(\delta,y^\delta) \geq \frac{\theta}{p 2^{p-1}} \frac{\tau^p-1}{\tau^p+1} \Phi_{q,w} ( (\tau-1) \delta ), \]
    where $\Phi_{q,w} (t) := t^p/\varphi_{q,w}(t)$.
  \end{corollary}

\begin{remark}
  To develop some intuition regarding the behaviour of the index functions $\varphi_{q,w} (t)$ from \eqref{eq:phiq} and the corresponding convergence rates in Corollary \ref{cor:cr}, we reconsider examples from \cite{BFH13, BoHo13} with uniform weights $w_k = 1$ for all $k \in \N$. Indeed, rates can be computed explicitely if the decay of the coefficients $|\xdags_k|$ and the growth of the norms $\| f_k \|_{Y^*}$ are of monomial type, i.e. if there exist $\mu > 1$, $\nu > 0$ and constants $K_1$, $K_2 > 0$ such that
    \[ |\xdags_k| \leq K_1 ~ k^{-\mu}, \qquad \| f_k \|_{Y^*} \leq K_2 ~ k^\nu, \]
  holds for all $k \in \supp (\xdag)$. Observe that $\xdag \in \M_q$ (cf. Assumption \ref{ass:basic} (a)) holds for all $1/\mu < q \leq 1$. Then taking
    \[ \sum_{k = n+1}^\infty |\xdags_k|^q \sim n^{-\mu q +1} \qquad \mbox{and} \qquad \sum_{k = 1}^n \| f_k \|_{Y^*}^q \sim n^{\nu q + 1}, \]
  yields an index function
    \[ \varphi_{q} (t) = t^{\frac{\mu q - 1}{\mu + \nu}}. \]
  The corresponding rates from Corollary \ref{cor:cr} are $\Reg_q (x_{\alpha_*}^\delta- \xdag) = \mathcal{O} (\delta^{\frac{\mu q - 1}{\mu + \nu}})$ and independent of the choice of $0<q \le 1$ in the penalty also
  $\Reg_1 (x_{\alpha_*}^\delta- \xdag) = \mathcal{O} (\delta^{\frac{\mu  - 1/q}{\mu + \nu}})$. As expected the faster the decay speed of $|\langle \xdag,u_k \rangle_X| \to 0$ as $k \to \infty$ expressed by the exponent $\mu$ is, the better is the rate. For the limit case $\mu \to \infty$ the obtained rates approximate from below $\mathcal{O}(\delta^{q})$ and for $q=1$ the linear rate.

\smallskip

  If the decay of the coefficients $|\xdags_k|$ is of exponential type in the sense that there exist $\gamma > 0$ and a constant $K_1 > 0$ such that
    \[ \sum_{k = n+1}^\infty |\xdags_k|^q \leq K_1 ~ e^{n^{-\gamma}} \]
  holds for all $n \in \N$, then for $f_k$ as in the previous example and by setting $n^\gamma \sim \log (\frac{1}{t^q})$ we find the resulting index function $\varphi_q$ to be at least of order
    \[ \varphi_q (t) \leq \mathcal{O} \left( t^q \log \Big( \frac{1}{t^q} \Big)^{\frac{\nu}{\gamma}} \right). \]
    Thus, the rate of convergence is slower than the rate $\mathcal{O} ( t^q )$ obtained for sparse solutions $\xdag \in \M_0$ by the logarithmic factor $\log ( \frac{1}{t^q} )^{\frac{\nu}{\gamma}}$, which is negligible if the exponent $\gamma$ is large.
\end{remark}

\section{Extensions} \label{sec:extensions}
\setcounter{equation}{0}
\setcounter{theorem}{0}

\subsection{Formulation in Banach spaces} \label{sec:ek}

  Some papers on $\ell^q$-regularization work with Banach spaces $X$ and exploit coefficients $\underline x_k$  with respect to a Schauder basis $\{u_k\}_{k \in \N}$ (see, e.g. \cite{BFH13}). In terms of such an infinite sequence $\underline x=(\underline x_1,\underline x_2,...)$ of coefficients and the linear \emph{synthesis operator} $L$ which transforms $\underline x$ to the element $x=L\underline x:=\sum \limits_{k=1}^\infty \underline x_k u_k \in X$, the Tikhonov functional reads
    \begin{equation} \label{eq:TikSeq}
      \underline T_\alpha^\delta(\underline x):=\frac{1}{p} \| \underline A \, \underline x-\yd\|_Y^p + \alpha \| \underline x \|_{q,w}^q,
    \end{equation}
  where $\underline A= A \circ L$.
  Clearly, each orthonormal basis in a Hilbert space $X$ is in particular a Schauder basis in $X$ with the specific property that there is a one-to-one correspondence between the elements $x \in X$ and the associated sequences of Fourier coefficients $\underline x \in \ell^2(\N)$. This is no longer the case if we consider Banach spaces that are not Hilbert spaces and for our purposes we only mention that $L:\1 \to X$ and hence the composition $\underline A= A \circ L: \1 \to Y$ are bounded linear operators. Note that the adjoint of $\underline A$ now maps as $\underline A^*: Y^* \to \3$.

  Then Assumption~\ref{ass:basic} can be rewritten by using the sequences $\underline e^{(k)}=(0,...,0,1,0,...)$ with $1$ in the $k$-th position of the sequence. Precisely, Item (a) attains the form

\smallskip

{\parindent0em $\hspace*{0.2cm} (a')\;\; \mbox{\sl Suppose that the uniquely determined solution} \;\, \xdag=L \xdags \in X\;\, \mbox{\sl of equation} \; (\ref{eq:opeq}) \\ \hspace*{1cm} \mbox{\sl with} \;\, y \in \range(A)\,\; \mbox{\sl satisfies the condition} \;\,\|\xdags\|_{q,w} < \infty\,\;\mbox{\sl for prescribed}\,\;0<q \le 1$.}

\smallskip

{\parindent0em Item} (b), however, has to be rewritten as

\smallskip

{\parindent0em  $\hspace*{0.2cm} (b')\;\; \mbox{\sl  For all } k \in \supp (\xdag) \mbox{ \sl there exist } f_k \in Y^* \mbox{\sl such that } \underline e^{(k)}=\underline A^*f_k, \mbox{\sl where } \underline A= A \circ L$.}

\smallskip

  \noindent Evidently, from Item $(a')$  it follows for all $0<q \le 1$  that $\xdags \in \1$. On the other hand, since $\{\underline e^{(j)}\}_{j \in \N}$ is a Schauder basis in $\1$ and  $u_j=L \,\underline e^{(j)}$, for each $k \in \supp (\xdag)$ condition $(b')$ asserts that
  \begin{equation} \label{eq:Krone1}
    \delta_{j,k} = \langle \underline e^{(k)}, \underline e^{(j)} \rangle_{\3 \times \1} = \langle \underline A^*f_k,\underline e^{(j)} \rangle_{\3 \times \1}
	= \langle f_k, A u_j \rangle_{Y^* \times Y}
  \end{equation}
  holds for all $j \in \N$, where $\delta_{j,k}$ denotes the Kronecker delta. Now, if $\{u_k\}_{k \in \N}$ is an orthonormal basis in a Hilbert space $X$, then (\ref{eq:Krone1}) is equivalent to
  \begin{equation} \label{eq:Krone2}
     \langle  A^*f_k, u_j  \rangle_X  = \delta_{j,k} \qquad \mbox{for all } j \in \N, 
  \end{equation}
  which yields $u_k=A^*f_k$ as required in Assumption~\ref{ass:basic} (b).

\subsection{Frames} \label{sec:frame}

   In the context of sparse regularization in a separable Hilbert space $X$ it could be advantageous to consider not only orthonormal bases, but to reconstruct coefficients with respect to a \emph{frame} in $X$. This is to say, a collection $\{u_k\}_{k \in \N} \subset X$ such that the \emph{frame condition}
    \begin{equation} \label{eq:frame}
     \underline b \| x\|_X^2 \leq \sum_{k \in \N} | \langle x, u_k \rangle |^2 \leq \overline b \| x\|_X^2
    \end{equation}
  holds with constants $0 < \underline b \leq \overline b < \infty$ (see, e.g., \cite{Chris03} for further details). If the frame is overcomplete which is the case whenever it is not a basis, then, on the one hand, more elements admit sparse or nearly sparse representations. But on the other hand, the nonuniqueness of the frame coefficients introduces additional difficulties. Indeed, the frame synthesis operator $L$ as defined in Section \ref{sec:ek} is bounded as $L: \2 \to X$ but not injective and hence neither is the composition $\underline A = A \circ L$. Thus $(b')$ from Section \ref{sec:ek} cannot be fulfilled (cf. \cite{BFH13}) which renders the results on convergence rates in Section \ref{sec:nonconvex} inapplicable when working with the Tikhonov functional $\underline T_\alpha^\delta$ from (\ref{eq:TikSeq}).

  This is not the case if we use the functionals $T_\alpha^\delta$ as defined in (\ref{eq:Tik}) where the forward operator under consideration is $A$ itself. At first glance, penalizing sequences $\underline x = L^* x = \{ \langle x, u_k \rangle_X \}_{k \in \N}$ may not seem rewarding as, in general, $\underline x$ here no longer consists of coefficients which reproduce $x \in X$ with respect to the frame $\{u_k\}_{k \in \N}$. It is, however, a sequence of coefficients with respect to the \emph{dual frames} of $\{u_k\}_{k \in \N}$ (cf. \cite[Section 5.6]{Chris03}), i.e. frames $\{ \tilde u_k \}$ such that
    \[ x = \sum_{k \in \N} \langle x, u_k \rangle_X \tilde u_k \qquad \mbox{for all } x \in X. \]
  Thus,  the minimizers $\xad$ of $T_\alpha^\delta$ have sparse decompositions with respect to the dual frames $\{ \tilde u_k\}_{k \in \N}$ and the convergence rates results in Section \ref{sec:nonconvex} apply with error measure $E(x, \xdag) = \Rqw (x - \xdag)$ whenever Assumption~\ref{ass:basic} is satisfied for the frame $\{u_k\}_{k \in \N}$ under consideration. Using the synthesis operator $\tilde L$ of the dual frame $\{ \tilde u_k\}_{k \in \N}$, we have
    \[ \| \xad - \xdag \|_X \leq \| \tilde L \| \cdot \| \underline x_\alpha^\delta - \underline x^\dag \|_2 \leq \| \tilde L \| \, \left(\frac{1}{w_0} \Rqw (\xad - \xdag)\right)^{1/q} \]
 and, consequently, obtain corresponding rates also in norm.

    It is worth noting that the proof of Theorem \ref{thm:vi} remains true without modification, if $\{ u_k \}_{k \in \N}$ is a frame in the Hilbert space $X$. However, in this case we would use the sparsity assumption $L^* \xdag \in \ell^q_w (\N)$ with respect to $\{ u_k \}_{k \in \N}$ in order to reconstruct coefficients with respect to some dual frame $\{ \tilde u_k \}_{k \in \N}$.

\subsection{Nonlinear operators} \label{sec:nonlin}

    Another possible generalization would be to consider also nonlinear forward operators $F(x)$ which are G\^ateaux differentiable and satisfy a structural condition
		\begin{equation}
		  \| F'(\xdag) (x - \xdag) \| \leq \sigma (\| F(x) - F(\xdag)\|), \qquad \mbox{for all } x \in M \subset X,
		\end{equation}		
    where $\sigma$ is a concave index function and $M \subset X$ should be large enough to contain all regularized solutions $\xad$ for sufficiently small $\delta > 0$ (cf.  \cite{BoHo10, BoHo13}). Then Assumption~\ref{ass:basic} (b) attains the form

\smallskip

{\parindent0em  $\hspace*{0.2cm} (b')\;\; \mbox{\sl  For all}\, \; k \in \supp (\xdag)\; \mbox{\sl there exist} \; f_k \in Y^*\; \mbox{\sl such that} \; u_k = F'(\xdag)^* f_k$,}

\smallskip

  {\parindent0em a} requirement previously used in \cite{Grasm09, BoHo13}, for example, and referred to as a \emph{curious} type of smoothness condition in \cite{BoHo13}. It is worth mentioning that from a technical point of view the results in the following sections remain applicable for such nonlinear operators $F(x)$, if the linearization $F'(\xdag)$ takes the place of the linear operator $A$ as necessary. The corresponding convergence rate function is given in Remark \ref{rem:nonlin}.

\section*{Acknowledgments}

% The authors express their thanks to Albrecht B\"ottcher and Arnd Meyer (TU Chemnitz) for helpful discussions and hints concerning model specifications in Sobolev spaces.
S.W.~Anzengruber and B.~Hofmann were supported by the German Science Foundation (DFG) under grant HO~1454/8-1. R. Ramlau was supported by the Austrian Science Fund, Project W1214.

% \section*{References}


\begin{thebibliography}{99}



\bibitem{AnzHofMath12}
S.W.~Anzengruber, B.~Hofmann, and P.~Math\'e.
\newblock Regularization properties of the discrepancy principle for Tikhonov
regularization in Banach spaces. Paper submitted, 2012.


\bibitem{AR1}
S.~W. Anzengruber and R.~Ramlau.
\newblock {M}orozov's discrepancy principle for {T}ikhonov-type functionals with nonlinear operators
\newblock {\em Inverse Problems}, 26(2):025001, 17pp., 2010.

\bibitem{AR2}
S.~W. Anzengruber and R.~Ramlau.
\newblock Convergence rates for {M}orozov's discrepancy principle using
  variational inequalities.
\newblock {\em Inverse Problems}, 27(10):105007, 18pp., 2011.


\bibitem{BakuKok04}
A.~B. Bakushinsky and M.~Yu. Kokurin.
\newblock {\em Iterative Methods for Approximate Solution of Inverse Problems}.
\newblock Springer, Dordrecht, 2004.

\bibitem{Bau87}
J.~Baumeister.
\newblock {\em Stable Solution of Inverse Problems.}
\newblock Vieweg, Braunschweig, 1987.




\bibitem{Bon09}
T.~Bonesky
\newblock Morozov's discrepancy principle and {T}ikhonov-type functionals
\newblock {\em Inverse Problems}, 25(1):015015, 11pp., 2009.



\bibitem{BoHo10}
R.~I. Bo\c{t} and B.~Hofmann.
\newblock An extension of the variational inequality approach for obtaining
  convergence rates in regularization of nonlinear ill-posed problems.
\newblock {\em Journal of Integral Equations and Applications}, 22(3):369--392,  2010.


\bibitem{BoHo13}
R.~I. Bo\c{t} and B.~Hofmann.
\newblock The impact of a curious type of smoothness conditions on
  convergence rates in $\ell^1$-regularization.
\newblock {\em Eurasian Journal of Mathematical and Computer Applications}, 1(1):29--40, 2013.


\bibitem{BreLor09}
K.~Bredies and D.~A. Lorenz.
\newblock Regularization with non-convex separable constraints.
\newblock {\em Inverse Problems}, 25(8):085011, 14pp., 2009.


\bibitem{BFH13}
M.~Burger, J.~Flemming and B.~Hofmann.
\newblock  Convergence rates in $\mathbf{\ell^1}$-regularization if the sparsity assumption fails.
\newblock {\em Inverse Problems}, 29(2):025013. 16pp, 2013.




\bibitem{Chris03}
O.~Christensen.
\newblock {\em An introduction to frames and Riesz bases}.
\newblock Birkh{\"a}user, Boston, 2003.


\bibitem{Daub03}
I.~Daubechies, M.~Defrise, and C.~De Mol.
\newblock An iterative thresholding algorithm for linear inverse problems with
  a sparsity constraint.
\newblock {\em Communications in Pure and Applied Mathematics},
  57(11):1413--1457, 2004.


\bibitem{EHN96}
H.W.~Engl, M.~Hanke, and A.~Neubauer.
\newblock {\em Regularization of Inverse Problems}. Volume 375 of {\em
  Mathematics and its Applications}.
\newblock Kluwer Academic Publishers Group, Dordrecht, 1996.

\bibitem{Flemmingbuch12}
J.~Flemming.
\newblock {\em {G}eneralized {T}ikhonov regularization and modern
  convergence rate theory in {B}anach spaces}.
\newblock Shaker Verlag, Aachen, 2012.

\bibitem{FHH13}
J.~Flemming, M.~Hegland, and B.~Hofmann.
\newblock Convergence rates in $\ell^1$-regularization when the basis is not smooth enough.
\newblock Paper in preparation, 2013.

\bibitem{Fle12}
J.~Flemming.
\newblock Solution smoothness of ill-posed equations in {H}ilbert spaces: four
  concepts and their cross connections.
\newblock {\em Appl. Anal.}, 91(5):1029--1044, 2012.

\bibitem{GorVes91}
R.~Gorenflo and S.~Vessella.
\newblock {\em Abel Integral Equations}.
Vol.~1461 of {\em Lecture Notes in Mathematics}.
\newblock Springer-Verlag, Berlin, 1991.

\bibitem{GorYam99}
R.~Gorenflo and M.~Yamamoto.
\newblock Operator-theoretic treatment of linear {A}bel integral
  equations of first kind.
\newblock {\em Japan J.~Indust.~Appl.~Math.}, 16(1):137--161, 1999.


\bibitem{Grasm09}
M.~Grasmair.
\newblock Well-posedness and convergence rates for sparse regularization
              with sublinear {$l\sp q$} penalty term.
\newblock {\em Inverse Probl. Imaging}, 3(3):383--387, 2009.

\bibitem{Gras10}
M.~Grasmair.
\newblock Non-convex sparse regularisation.
\newblock {\em J.~Math.~Anal.~Appl.}, 365(1):19--28, 2010.

\bibitem{Grasm10}
M.~Grasmair.
\newblock Generalized {B}regman distances and convergence rates for non-convex
  regularization methods.
\newblock {\em Inverse Problems}, 26(11):115014, 16pp., 2010.





\bibitem{GrasHaltSch08}
M.~Grasmair, M.~Haltmeier, and O.~Scherzer.
\newblock Sparse regularization with {$l\sp q$} penalty term.
\newblock {\em Inverse Problems}, 24(5):055020, 13pp., 2008.





\bibitem{GreNg13}
D.~S.~Grebenkov and B.-T.~Nguyen.
\newblock Geometrical structure of Laplacian eigenfunctions. Preprint, http://arxiv.org/pdf/1206.1278.pdf, 2013.

\bibitem{HeinHof09}
T.~Hein and B.~Hofmann.
\newblock Approximate source conditions for nonlinear ill-posed
  problems---chances and limitations.
\newblock {\em Inverse Problems}, 25(3):035003, 16pp., 2009.





\bibitem{HKPS07}
B.~Hofmann, B.~Kaltenbacher, C.~P\"{o}schl, and O.~Scherzer.
\newblock A convergence rates result for {T}ikhonov regularization in {B}anach
  spaces with non-smooth operators.
\newblock {\em Inverse Problems}, 23(3):987--1010, 2007.

\bibitem{HofMat12}
B.~Hofmann and P.~Math\'e.
\newblock Parameter choice in {B}anach space regularization under variational
  inequalities.
\newblock {\em Inverse Problems}, 28(10):104006, 17pp, 2012.





\bibitem{JinMaass12}
B.~Jin and P.~Maass.
Sparsity regularization for parameter identification problems.
\newblock  {\em Inverse Problems}, 28(12):123001, 70pp, 2012.


\bibitem{KMS12}
K.~S.~Kazimierski, P.~Maass,  and R.~Strehlow.
\newblock Norm sensitivity of sparsity regularization with respect to {$p$}.
\newblock {\em Inverse Problems}, 28(10):104009, 20pp, 2012.


\bibitem{Kirsch11}
A.~Kirsch.
\newblock {\em An Introduction to the Mathematical Theory of Inverse Problems}. Volume 120 of {\em Applied Mathematical Sciences} (2nd.~Edition).
\newblock Springer, New York, 2011.


\bibitem{Lorenz08}
D.~A. Lorenz.
\newblock Convergence rates and source conditions for {T}ikhonov regularization
  with sparsity constraints.
\newblock {\em J. Inverse Ill-Posed Probl.}, 16(5):463--478, 2008.







\bibitem{Mey}
Y.~Meyer.
\newblock  {\em Wavelets and Operators}. Volume 37 of {\em Cambridge Studies in Advanced Mathematics}.
\newblock  Cambridge University Press, 1992.


\bibitem{n86}
F. Natterer.
\newblock {\em The Mathematics of Computerized Tomography}.
\newblock Teubner, Stuttgart, 1986.

\bibitem{nw01}
F. Natterer, F. W{\"u}bbeling.
\newblock {\em Mathematical methods in image reconstruction}. Series: {\em SIAM monographs on mathematical modeling and computation}.
\newblock SIAM, Philadelphia, PA, 2001.







\bibitem{RamZar12}
R.~Ramlau and C.~A.~Zarzer.
\newblock On the minimization of a Tikhonov functional with a non-convex sparsity constraint.
\newblock {\em Electron.~Trans.~Numer.~Anal.}, 39:476--507, 2012.


\bibitem{Rieder03}
A.~Rieder.
\newblock Keine Probleme mit inversen Problemen.
\newblock Vieweg \& Sohn, Braunschweig, 2003.


\bibitem{Scherzetal09}
O.~Scherzer, M.~Grasmair, H.~Grossauer, M.~Haltmeier, and F.~Lenzen.
\newblock {\em {V}ariational {M}ethods in {I}maging}. Volume 167 of {\em
  Applied Mathematical Sciences}.
\newblock Springer, New York, 2009.


\bibitem{SchKHK12}
T.~Schuster, B.~Kaltenbacher, B.~Hofmann, and K.S. Kazimierski.
\newblock {\em Regularization Methods in Banach Spaces}. Volume~10 of {\em
Radon Ser. Comput. Appl. Math.}
\newblock Walter de Gruyter, Berlin/Boston, 2012.







\bibitem{Wangetal13}
W.~Wang, S.~Lu, Mao, H., and J.~Cheng.
\newblock Multi-parameter Tikhonov regularization with $\ell^0$ sparsity constraint.
\newblock {\em Inverse Problems} 29, 2013 (to appear).





\bibitem{Xuetal10}
Z.~Xu, H.~Zhang, Y.~Wang, X.Y.~Chang, and Y.~Liang.
\newblock $L\sb {1/2}$ regularization.
\newblock {\em Sci.~China Inf.~Sci.}, 53:1159--1169, 2010.


\bibitem{Yos80}
K.~Yosida.
\newblock {\em Functional Analysis}. Volume~123 of
{\em Fundamental Principles of Mathematical Sciences} (6th ed.)
\newblock Springer-Verlag, Berlin, 1980.


\bibitem{Zar09}
C.A.~Zarzer.
\newblock On {T}ikhonov regularization with non-convex sparsity constraints.
\newblock {\em Inverse Problems}, 25:02006, 13pp, 2009.



\end{thebibliography}
\end{document}